\documentclass[12pt, twoside]{article}
\usepackage{amsmath,amsthm,amssymb}
\usepackage{times}
\usepackage{enumerate}
\usepackage{bm}
\usepackage[all]{xy}
\usepackage{mathrsfs}
\usepackage{amscd}
\usepackage{color}
\def\titlerunning#1{\gdef\titrun{#1}}
\makeatletter
\def\author#1{\gdef\autrun{\def\and{\unskip, }#1}\gdef\@author{#1}}
\def\address#1{{\def\and{\\\hspace*{18pt}}\renewcommand{\thefootnote}{}%
\footnote {#1}}%
\markboth{\autrun}{\titrun}}
\makeatother
\def\email#1{e-mail: #1}
\def\subjclass#1{{\renewcommand{\thefootnote}{}%
\footnote{\emph{Mathematics Subject Classification (2020):} #1}}}
\def\keywords#1{\par\medskip
\noindent\textbf{Keywords.} #1}

\newtheorem{theorem}{Theorem}[section]
\newtheorem{corollary}[theorem]{Corollary}
\newtheorem{lemma}[theorem]{Lemma}
\newtheorem{proposition}[theorem]{Proposition}
\theoremstyle{definition}
\newtheorem{definition}[theorem]{Definition}
\newtheorem{remark}[theorem]{Remark}

\numberwithin{equation}{section}

\xyoption{all}

\def \a {\alpha }
\def \b {\beta}

\def \La {\Lambda}

\def\Om{\Omega}

\def\Ga{\Gamma}

\begin{document}
\baselineskip=17pt

\titlerunning{Gårding cones and positivity of curvature operators}
\title{Gårding cones and positivity of curvature operators}

\author{Teng Huang and Jiaogen Zhang }

\date{}

\maketitle

\address{T. Huang: School of Mathematical Sciences, University of Science and Technology of China; CAS Key Laboratory of Wu Wen-Tsun Mathematics, University of Science and Technology of China, Hefei, Anhui, 230026, P. R. China; \email{htmath@ustc.edu.cn;htustc@gmail.com}}
\address{J. Zhang: School of Mathematics, Hefei University of Technology,
   Hefei, Anhui, 230009, P. R. China;
\email{zjgmath@ustc.edu.cn}}
\subjclass{53C20; 53C21; 58A10; 58A14}

\begin{abstract}
This article explores the relationship between Gårding cones $\Gamma^{+}_{k}$ and $k$-positive cone $\mathcal{P}_{k}$, demonstrating that under particular choices of $m$ and $\alpha$, the shift cone $\overline{\Ga}^{+}_{2}(\alpha)$ is contained in $\overline{\mathcal{P}}_{m}$. By combining these results with the study of positivity properties of curvature operators, we establish several new connections between algebraic positivity conditions and the geometry of underlying Riemannian manifolds. Our main theorems reveal how shifted cone conditions on curvature operators—both standard and of the second kind—constrain topology, including vanishing theorems for Betti numbers and characterizations of spherical space forms. 
\end{abstract}
\keywords{Gårding cones,  curvature operators,  Betti numbers, spherical space forms}

\section{Introduction}

A central topic in Riemannian geometry is understanding how geometric assumptions constrain the topology of underlying manifolds, it is a  fundamental problem to give reasonable curvature conditions of positivity to ensure a topological classification of such manifolds. A seminal result in this direction was established by Meyer \cite{Mey71}, who proved that compact manifolds admitting positive curvature operators must be rational homology spheres. This work was later complemented by Gallot and Meyer \cite{GM75}, who demonstrated the  corresponding rigidity theorem. Namely, they proved that every compact manifold  with nonnegative curvature operator is either reducible, locally symmetric, or its universal cover has the cohomology of a sphere or a complex projective space. 
 
In a parallel development, Ogiue and Tachibana \cite{OT79} established an analogous result to Meyer's theorem, proving that compact manifolds with positive curvature operators of the \textit{second kind} are exactly rational homology spheres. This led Nishikawa \cite{Nis86} to further conjecture that such manifolds are diffeomorphic to spherical space forms. A major breakthrough was achieved by Cao, Gursky, and Tran \cite{CGT23} recently, who confirmed Nishikawa's conjecture for compact manifolds with merely $2$-positive curvature operators of the second kind. Subsequently, Li \cite{Li24} generalized their result by relaxing the assumption to $3$-positive on the curvature operators of the second kind.

Over the past four decades, Ricci flow has become a cornerstone technique for classifying compact manifolds with positive curvature operators, leading to significant results such as the well-known Bochner vanishing-type theorems \cite{Boc46}. For instance, a seminal theorem by Hamilton \cite{Ham82,Ham86} established that compact manifolds endowed with  positive curvature operators must be diffeomorphic to space forms. This result was later  generalized by Chen \cite{Che91} and B\"{o}hmann and Wilking \cite{BW08} to manifolds with only $2$-positive curvature operators. The corresponding rigidity theorem was subsequently resolved by Ni and Wu \cite{NW07}.  Brendle \cite{Bre08} and Brendle and Schoen \cite{BS08,BS09} established that $M\times \mathbb{R}^{i}$ $(i=1,2)$ have positive isotropic curvature, leading to the proof of the differentiable sphere theorem. Namely,  any manifold with $1/4$-pinched  sectional curvature admits a metric of constant scalar curvature and is therefore diffeomorphic to a spherical space form. For further classifications on compact manifolds with positive isotropic curvature, we refer to \cite{Bre19,CTZ12,CZ06,Ham97,Huang23,MM88} and the reference therein.

Betti numbers are fundamental topological invariants that distinguish between spaces that are not homotopy equivalent. Let $(X,g)$ be a compact $n$-dimensional Riemannian manifold without boundary.  It is an interesting problem to investigate under what geometric conditions the $p$-th Betti number $b_{p}(X)$ $(1\leq p\leq n-1)$ vanishes. According to  Meyer \cite{Mey71} and Gallot and Meyer \cite{GM75}, it was established that the $b_{p}(X)=0$  for all $p$ under the condition that the curvature operator is nonnegative and positive at least at one point.  By assuming the curvature operator is merely $(n-p)$-positive, Petersen and Wink \cite{PW21a} established that $b_{q}(X)=b_{n-q}(X)=0$ for all $q\in \{1,\cdots,p\}$ when $1\leq p\leq \lfloor \frac{n}{2}\rfloor$\footnote{In this article, $\lfloor c\rfloor$ denotes the largest integer not exceeding $c$, and $\lceil c\rceil$ denotes the smallest integer not less than $c$.}. In particular, when the curvature operator is $\lceil \frac{n}{2} \rceil$-positive, we have $b_{p}(X)=0$ for all $p$. Subsequently, in a collaboration with Nienhaus in \cite{NPW23}, they further demonstrated that  for all integers $1\leq p\leq \frac{n}{2}$, the condition $b_{p}(X)=0$ holds when the curvature operator of the second kind is $C_{p}$-positive.  Recently, Yang and Zhang \cite{YZ25} introduced a new positivity concept for curvature operators, which implies the vanishing properties of Betti numbers. 

More specifically, let $\Lambda_{1}=(\lambda_{1},\cdots,\lambda_{N})$ with $N=\frac{n(n-1)}{2}$ be the eigenvalues of curvature operator $\mathfrak{R}$ on $\Lambda^2TX$. Define $\alpha_{k}=\frac{1}{N}\big(1-\sqrt{\frac{k}{(N-1)(N-k)}}\big)$ for every $k=1,\cdots,N-1$. The condition $\Lambda_1\in \Gamma_{2}^{+}(\alpha_{k})$ (see Definition \ref{Gamma k alpha} below) ensures that $b_{p}(X)=0$ for all $p$ when $1\leq k\leq \lfloor \frac{n}{2}\rfloor$, and  $b_{p}(X)=b_{n-p}(X)=0$ for all $p\in \{1,\cdots,n-k\}$ when  $\lfloor \frac{n}{2}\rfloor+1 \leq k\leq n-1$. We remark that the main result in \cite{YZ25} relies critically on the cone inclusion relation $\Gamma_{2}^{+}(\alpha_{k}) \subset \mathcal{P}_{k}$, where  $\mathcal{P}_{k}$ is defined as in \eqref{Pk} below.  In this work, we  establish a more general cone inclusion relation between $\Gamma_{2}^{+}(\alpha)$ and $\mathcal{P}_{m}$ (see Definition \ref{Pm} below) for arbitrary $0<\alpha<\frac{1}{N}$, broadening the applicability of these results in differential geometry. 

Our first main result is the following
\begin{theorem}\label{T1}
For every $0<\epsilon<1$, let
$$\a_{\epsilon}=\frac{1}{N}(1-\epsilon)\quad and\quad m_{\epsilon} =\frac{N(N-1)\epsilon^2}{1+(N-1)\epsilon^2}.$$
We have the following cone inclusion relation:
\begin{equation*}
\overline{\Gamma}^{+}_{2}(\alpha_{\epsilon})\subset \overline{\mathcal{P}}_{m_{\epsilon}}\,.
\end{equation*}
Moreover, for every vector $V=(\mu_{1},\cdots,\mu_{N})\in \overline{\Gamma}^{+}_{2}(\alpha_{\epsilon})$ with $\mu_{1}\leq\cdots\leq\mu_{N}$ and $\bar{\mu}:=\sum_{i=1}^{n}\mu_{i}>0$, then one of the following holds:
\begin{itemize}
\item[(1)] We have $\mu_{1}+\cdots+\mu_{\lfloor  m_{\epsilon}\rfloor}+\big(m_{\epsilon}-\lfloor m_{\epsilon}\rfloor\big)\mu_{\lfloor m_{\epsilon}\rfloor+1}> 0$.
\item[(2)] We have $m_{\epsilon}\in\mathbb{N}^{+}$, $\mu_{1}=\mu_{2}=\cdots=\mu_{ m_{\epsilon} }=0$, and $\mu_{ m_{\epsilon} +1}=\mu_{  m_{\epsilon} +2}=\cdots=\mu_{N}>0$.
\end{itemize}
\end{theorem}
\begin{corollary}\label{C2}
For every $0<\epsilon<1$, we have
\begin{equation*}
\Gamma^{+}_{2}(\alpha_{\epsilon})\subset {\mathcal{P}_{m_{\epsilon}}}\,.
\end{equation*} 
\end{corollary}
\begin{remark}
Theorem \ref{T1} extends the result of Yang and Zhang \cite[Lemma 2.4]{YZ25}, showing that when $\epsilon = \sqrt{\frac{k}{(N-1)(N-k)}}$ (i.e. $m_{\epsilon} = k$ ) for some integer $1\leq k\leq N-1$. Compared to theirs, our proof is not only simpler but also applies to arbitrary $\epsilon \in (0,1)$.
\end{remark}

\begin{definition}\label{shifted cone}(Shifted vectors) 
Denote
\begin{equation}\label{N12}
N_{1}=\frac{n(n-1)}{2} \quad \textrm{and} \quad N_{2}=\frac{(n-1)(n+2)}{2}\,.
\end{equation} 
Consider vectors $\Lambda_{1} = (\lambda_{1}, \dots, \lambda_{N_{1}})\in \mathbb{R}^{N_1}$ and   $\Lambda_{2} = (\nu_{1}, \dots, \nu_{N_{2}})\in \mathbb{R}^{N_2}$. We say $\Lambda_1$ is \textit{ordered} if $\lambda_1   \leq \cdots \leq \lambda_{N_1}$, and $\Lambda_2$ is \textit{ordered} if $\nu_1 \leq \cdots \leq \nu_{N_2}$. For fixed nonnegative constants $\alpha\in (0,\frac{1}{N_{1}})$ and $\beta\in (0,\frac{1}{N_{2}})$, define the shifted vectors
\begin{equation}
\Lambda_{\alpha} := \Lambda_{1} - \alpha \big(\sum_{i=1}^{N_{1}}\lambda_{i}, \dots, \sum_{i=1}^{N_{1}}\lambda_{i}\big), \qquad \Lambda_{\beta} := \Lambda_{2} - \beta \big(\sum_{i=1}^{N_{2}}\nu_{i}, \dots, \sum_{i=1}^{N_{2}}\nu_{i}\big).
\end{equation}
\end{definition}

In the second part of this article, we apply Theorem \ref{T1} to two classes of curvature operators in Riemannian geometry: the curvature operator and the curvature operator of the second kind.

First, the curvature operator defined as
$$\mathfrak{R}:\Lambda^{2}TX\rightarrow\Lambda^{2}TX,\quad (\mathfrak{R}(\omega))_{ij}=\sum_{k,l=1}^{n}R_{ijkl}\omega_{kl}.$$
Suppose that $\Lambda_{1}=(\lambda_{1},\cdots,\lambda_{N_1})\in \mathbb{R}^{N_1}$ are the  ordered  eigenvalues of the curvature operator $\mathfrak{R}$ on $\Lambda^{2}TX$. The scalar curvature $R$ of the Riemannian manifold $(X,g)$ can be expressed in terms of these eigenvalues as
\begin{equation}\label{R1}
R=2\sum_{i=1}^{N_1}\lambda_{i}.
\end{equation}

Beyond the standard curvature operator, the Riemann curvature tensor naturally induces a self-adjoint operator on the space of symmetric $(0,2)$-tensors:
$$\overline{R}:S^{2}(TX)\rightarrow S^{2}(TX), \qquad (\overline{R}(h))_{ij}=\sum_{k,l}R_{iklj}h_{kl}.$$
The curvature operator of the second kind is the induced map on the space of trace-free symmetric $(0,2)$-tensors:
$$\mathcal{R}:S_{0}^{2}(TX)\rightarrow S^{2}_{0}(TX), \qquad \mathcal{R}={\rm{pr}}_{S^{2}_{0}(TX)}\circ\overline{R}|_{S^{2}_{0}(TX)}.$$
This operator was previously studied by Bourguignon and Karcher \cite{BK78}. In contrast to $\overline{R}$, the curvature operator of the second kind $\mathcal{R}$ satisfies the natural geometric property that $\mathcal{R}\geq \kappa$ implies all sectional curvatures are bounded below by $\kappa$.
Suppose that $\Lambda_{2}=(\nu_{1},\cdots,\nu_{N_2})\in \mathbb{R}^{N_2}$ are the ordered eigenvalues of $\mathcal{R}$ on $S^{2}_{0}(TX)$. In this case, the scalar curvature $R$ of the Riemannian manifold $(X,g)$ can be expressed in terms of these eigenvalues as
\begin{equation}\label{R2}
R=\frac{2n}{n+2}\sum_{i=1}^{N_2}\nu_{i}.
\end{equation}
 
To study the prescribed $\sigma_{k}$-scalar curvature problem, Viaclovsky \cite{Via00} introduced the $k$-th elementary symmetric function of the Schouten tensor's eigenvalues. Subsequently, many significant contributions have been made to this area; see e.g. \cite{GLW04,GLW05,GVW03,GW04} and references therein. Inspired by this framework, Yang and Zhang \cite{YZ25} introduced a new positivity concept for the curvature operator $\mathfrak{R}$. In this paper, we also extend this positivity concept to the curvature operator of the second kind $\mathcal{R}$.
\begin{definition}
For a fixed index $j\in\{1,2,\cdots,N_{1} \}$, we say that the metric $g\in \Gamma^{+}_{j}(\alpha)$ (resp. $g\in\overline{\Ga}^{+}_{j}(\alpha)$) if the shifted vector $\Lambda_{\alpha}\in\Gamma^{+}_{j}$ (resp. $\Lambda_{\alpha}\in\overline{\Ga}^{+}_{j}$) holds at every point $x\in X$.

Analogously, for a fixed index $j\in\{1,2,\cdots,N_{2} \}$, we say that the metric $g\in \Gamma^{+}_{j}(\beta)$ (resp. $g\in\overline{\Ga}^{+}_{j}(\beta)$) if the shifted vector $\Lambda_{\beta}\in\Gamma^{+}_{j}$ (resp. $\Lambda_{\beta}\in\overline{\Ga}^{+}_{j}$) holds at every point $x\in X$.
\end{definition}

Building upon their works \cite{NPW23, NPWW23, PW21a, PW21b}, the authors systematically developed the theory of curvature operators, introducing both: the nonnegativity of the curvature operator, and its counterpart for curvature operators of the second kind.
\begin{definition}
Let $(X,g)$ be an $n$-dimensional Riemannian  manifold. Suppose that the curvature operator $\mathfrak{R}$ having ordered eigenvalues $\Lambda_{1}=(\lambda_{1},\cdots,\lambda_{N_1})$, and the curvature operator of the second kind $\mathcal{R}$ having ordered eigenvalues $\Lambda_{2}=(\nu_{1},\cdots,\nu_{N_2}) $. 
We say that 
\begin{itemize}
\item[(i)] $\mathfrak{R}$ is $m$-positive (resp. $m$-nonnegative) with $1\leq m\leq N_{1}-1$ if
$$\lambda_{1}+\lambda_{2}+\cdots\lambda_{\lfloor m\rfloor}+(m-\lfloor m\rfloor)\lambda_{\lfloor m\rfloor+1}>0\quad (\textrm{resp.} \geq 0) \,.$$
\item[(ii)] $\mathcal{R}$ is $m$-positive (resp. $m$-nonnegative) with $1\leq m\leq N_{2}-1$ if
$$\nu_{1}+\nu_{2}+\cdots\nu_{\lfloor m\rfloor}+(m-\lfloor m\rfloor)\nu_{\lfloor m\rfloor+1}>0\quad (\textrm{resp.} \geq 0)\,.$$
\end{itemize}
\end{definition}
\begin{remark}
We observe that for a Riemannian manifold, the following conditions are equivalent:
\begin{itemize}
\item[(i)] The curvature operator $\mathfrak{R}$ (resp. the curvature operator of the second kind $\mathcal{R}$) is $m$-positive.
\item[(ii)] The corresponding eigenvalue vector $\Lambda_1=(\lambda_{1},\cdots,\lambda_{N_{1}})\in\mathcal{P}_{m}$ 
 (resp. $\Lambda_2=(\nu_{1},\cdots,\nu_{N_{2}})\in\mathcal{P}_{m}$).
\end{itemize}
\end{remark}

For any $\epsilon\in(0,1)$, define 
\begin{equation}
\alpha_\epsilon=\beta_\epsilon=\frac{1}{N}(1-\epsilon) \quad and\quad
m_{\epsilon}= \frac{N(N-1)\epsilon^2}{1+(N-1)\epsilon^2},
\end{equation}
where $N=N_{1}$ for the curvature operator $\mathfrak{R}$ and $N=N_{2}$ for the curvature operator of the second kind $\mathcal{R}$. As a first application of Theorem \ref{T1}, we establish the following vanishing result for Betti numbers under a certain curvature constraint  on the operator  $\mathfrak{R}$ or $\mathcal{R}$.
\begin{theorem}(Theorem \ref{T2}+Theorem \ref{T4})
Let $(X,g)$ be a compact Riemannian manifold of dimension $n\geq 3$. If $g\in\Gamma^{+}_{2}(\a_{\epsilon})$ (resp. $g\in\Gamma^{+}_{2}(\beta_{\epsilon})$) for some $\epsilon\in(0,1)$, then its curvature operator $\mathfrak{R}$ (resp. the curvature of the second kind $\mathcal{R}$) is $m_{\epsilon}$-positive. Furthermore, 
\begin{itemize}
\item[(1)] For the curvature operator $\mathfrak{R}$,
\begin{itemize}
\item If $k-1< m_{\epsilon}\leq k$ for an integer $1\leq k\leq \lceil\frac{n}{2}\rceil$, then $b_{p}(X)=0$ for all $1\leq p\leq n-1$.
\item If  $k-1<m_{\epsilon}\leq k$ for an integer $ \lceil\frac{n}{2}\rceil+1\leq k\leq n-1$, then 
$$ b_{1}(X)=\cdots=b_{n-k}(X)=0\quad and\quad b_{k}(X)=\cdots=b_{n-1}(X)=0.$$
\end{itemize}
\item[(2)] For the curvature operator of second kind $\mathcal{R}$,
\begin{itemize}
\item If  $m_{\epsilon}\leq\lfloor\frac{3n}{4}\rfloor$, then $b_{p}(X)=0$ for all $1\leq p\leq n-1$.
\item If $m_{\epsilon}\leq C_{k}(n):=\frac{3}{2}\frac{n(n+2)k(n-k)}{n^{2}k-nk^{2}-2nk+2n^{2}+2n-4k}$ and $1\leq k\leq\lfloor\frac{n}{2}\rfloor$, then 
$$b_{k}(X)=\cdots=b_{2n-k}(X).$$
\end{itemize}
\end{itemize}
\end{theorem}
As a second application of Theorem \ref{T1}, we characterize manifolds diffeomorphic to spherical space forms.
\begin{theorem}(Theorem \ref{T3}+Theorem \ref{T5})
Let $(X,g)$ be a compact Riemannian manifold of dimension $n\geq 3$.
\begin{itemize}
\item[(1)] For the curvature operator $\mathfrak{R}$. If $g\in \Gamma^{+}_2(\alpha_\epsilon)$ with
$0<\epsilon\leq\sqrt{\frac{8}{(n^{2}-n-2)(n^{2}-n-4)}}$, then $X$ is diffeomorphic to a spherical space form.
\item[(2)] For the curvature operator of the second kind $\mathcal{R}$. If $g \in \Gamma^{+}_2(\beta_\epsilon)$ with $0<\epsilon\leq\sqrt{\frac{12}{(n^{2}+n-4)(n^{2}+n-8)}}$, then $X$ is diffeomorphic to a spherical space form.
\end{itemize}
\end{theorem}

\section{Positively of Gårding cone}
\subsection{Gårding cones and its properties}
Let $\Gamma^{+} \subset \mathbb{R}^{N}$ be a nonempty open convex cone with vertex at the origin, and let $\overline{\Gamma}^{+}$ denote its closure. For any nonzero vector $V=(\mu_{1},\cdots,\mu_{N})\in \mathbb{R}^{N}$, we consider the elementary symmetric functions $\sigma_{k}$, $(1\leq k\leq N)$  defined by:
\begin{equation}\label{kth elementary symmetric}
\sigma_{k}(V)=\sum_{1\leq i_1<\cdots< i_{k}\leq N}\mu_{i_1} \cdots  \mu_{i_{k}}.
\end{equation}
In particular, we note the special cases:
\begin{itemize}
\item[(i)] $\sigma_{1}(V)=\mu_{1}+\cdots+\mu_{N}$ (the trace);
\item[(ii)]$\sigma_{n}(V)=\mu_{1}\cdots \mu_{N}$ (the determinant);
\item[(iii)] The quadratic form can be expressed as:
\begin{equation}\label{sigma 2 definition}
\sigma_{2}(V)=\Big(\sum_{i=1}^{N}\mu_{i}\Big)^{2}-\sum_{i=1}^{N}\mu_{i}^{2}.
\end{equation}
\end{itemize}
\begin{definition}(Gårding cones)
For every $1\leq k\leq\mathbb{N}$, we define the Gårding cone:
$$\Gamma^{+}_{k}:=\{ V\in \mathbb{R}^N \,: \,\sigma_{j}(V)>0 \textrm{ for all } j=1,\cdots,k\}.$$
\end{definition}
These cones satisfy the nested inclusion relation:
\begin{equation}\label{Gamma inclusion}
\Gamma^{+}_{N}\subset \cdots \subset \Gamma^{+}_{1}\,.
\end{equation}
Observe that $\Gamma^{+}_{N}$ is the positive orthant, whereas $\Gamma^{+}_{1}$ is a half-space in $\mathbb{R}^{N}$ containing $\Gamma^{+}_{N}$. It is well established that $\Gamma^{+}_{k}$ is an open convex cone in $\mathbb{R}^N$ and that the operator defined by the $k$-th root $\sigma_{k}^{1/k}$ is concave on $\Gamma^{+}_{k}$.

For a given constant $\alpha\geq 0$ and any vector $V=(\mu_{1},\cdots,\mu_{N})\in\mathbb{R}^{N}$, we define its $\a$-shift as
\begin{equation}\label{V alpha}
V_{\alpha}=V-\alpha\,\overline{\mu}\,I_{N}=V-\alpha\,(\overline{\mu},\cdots,\overline{\mu})
\end{equation}
where $\overline{\mu}=\sum_{i=1}^{N}\mu_{i}$ and $I_N=(1,...,1)$ is the unit vector in $\mathbb{R}^{N}$. The $\a$-perturbed Gårding cone is defined by:
\begin{equation}\label{Gamma k alpha}
\Gamma^{+}_{k}(\alpha): =\{V\in\mathbb{R}^{N}\,:\, V_{\alpha}\in \Gamma^{+}_k\}.
\end{equation}
\begin{remark}
\quad\\
(1) When $\a=0$, we recover the original Gårding cones: $\Gamma^{+}_{k}(0)=\Gamma^{+}_{k}$ for every $1\leq k\leq N$.\\
(2) For the perturbed cones, we also have following inclusion chain which analogous to the inclusion chain \eqref{Gamma inclusion}:
\begin{equation}\label{Gamma alpha inclusion}
\Gamma^{+}_{N}(\alpha)\subset \cdots \subset \Gamma^{+}_{1}(\alpha)\,.
\end{equation}
\end{remark}
We now consider another Gårding-type cone
\begin{equation}\label{Pk}
\mathcal{P}_{k}:=\{ V\in \mathbb{R}^N \,: \,\mu_{i_1}+\cdots +\mu_{i_{k}}>0 \textrm{ for all } 1\leq i_1<\cdots< i_{k}\leq N\}.
\end{equation}
 
\begin{remark}
When $V$ represents the eigenvalues of a matrix in $\mathbb{R}^N$:\\
(1) $V\in\mathcal{P}_{k}$ implies at least $(N-k+1)$ positive eigenvalues;\\
(2) $V\in\overline{\mathcal{P}}_{k}$ implies at least $(N-k+1)$ non-negative eigenvalues.
\end{remark}
We also have the following inclusion relations:
\begin{equation}\label{P inclusion}
\Gamma^{+}_{n}=\mathcal{P}_{1}\subset \mathcal{P}_{2}\subset \cdots \mathcal{P}_{N-1}\subset \mathcal{P}_{N}=\Gamma^{+}_{1}\,.
\end{equation}
The aforementioned cone family, commonly referred to as $k$-\textit{convex cone}, was first introduced by Harvey and Lawson \cite{HL13} in their seminal work on $\mathbb{G}$-plurisubharmonicity theory, where they established fundamental results using this framework.
\begin{definition}\label{Pm}
($m$-positivity cone).
For any real parameter $1\leq m\leq N$ (not necessarily integer), we define the $m$-positivity cone:
$$\mathcal{P}_{m}=\{V\in \mathbb{R}^N :  \mu_{i_1}+\cdots +\mu_{i_{ \lfloor m\rfloor}}+(m-\lfloor m\rfloor)\mu_{i_{\lfloor m\rfloor+1}}>0 \textrm{ for all }  1\leq i_1<\cdots< i_{\lfloor m\rfloor+1}\leq N\}.$$
\end{definition}
For integer values $m=k\in\mathbb{N}$, the definition reduces to the standard $k$-\textit{convex cone} $\mathcal{P}_{k}$. The following inclusion relations hold:
\begin{itemize}
\item[(i)]  $\mathcal{P}_{\lfloor m\rfloor}\subset \mathcal{P}_{m}$ for all $1\leq m\leq N$;
\item[(ii)]  $\mathcal{P}_{m}\subset \mathcal{P}_{\lfloor m\rfloor+1}$ for all $1\leq m<N$;
\item[(iii)] $\mathcal{P}_{m_{1}}\subset\mathcal{P}_{m_{2}}$ for any $1\leq m_{1}\leq m_{2}\leq N$ (cf. Lemma \ref{L1} below).
\end{itemize}
The closure $\overline{\mathcal{P}}_{m}$, called the $m$-nonnegativity cone, satisfies analogous properties.
\begin{lemma}\label{L1}
Let $1\leq k_{1}< k_{2}\leq N$ and $\mu_{1}\leq \cdots\leq \mu_{N}$. It holds that
\begin{equation*}
\frac{1}{k_{1}}\Big(\sum_{i=1}^{\lfloor k_{1} \rfloor}\mu_{i}+(k_1-\lfloor k_{1} \rfloor)\mu_{\lfloor k_{1} \rfloor+1} \Big)\leq  \frac{1}{k_{2}}\Big(\sum_{i=1}^{\lfloor k_{2} \rfloor}\mu_{i}+(k_2-\lfloor k_{2} \rfloor)\mu_{\lfloor k_{2} \rfloor+1} \Big)\,.
\end{equation*}
\end{lemma}\begin{proof}
After restating the original inequality as
\[
k_{2}\Big(\sum_{i=1}^{\lfloor k_{1} \rfloor}\mu_{i}+(k_1-\lfloor k_{1} \rfloor)\mu_{\lfloor k_{1} \rfloor+1} \Big)\leq  k_{1}\Big(\sum_{i=1}^{\lfloor k_{2} \rfloor}\mu_{i}+(k_2-\lfloor k_{2} \rfloor)\mu_{\lfloor k_{2} \rfloor+1} \Big)\,.\]

\noindent \textbf{Case 1.} Assume $\lfloor k_{1} \rfloor=\lfloor k_{2} \rfloor=k$. 
The inequality reduces to
\[
(k_2-k_1)\sum_{i=1}^{k}\mu_{i}\leq (k_2-k_{1})k\mu_{k+1}.
\]
This holds trivially due to the non-increasing property of $\{\mu_{i}\}_{i=1}^{N}$.\medskip

\noindent \textbf{Case 2.} When $\lfloor k_2]\geq \lfloor k_{1} \rfloor+1$, it suffices to prove
\begin{equation}\label{monotone}
k_{1}\sum_{i=\lfloor k_{1} \rfloor+1}^{\lfloor k_{2} \rfloor}\mu_{i}+k_{1}(k_2-\lfloor k_{2} \rfloor)\mu_{\lfloor k_{2} \rfloor+1}\geq (k_2-k_1)\sum_{i=1}^{\lfloor k_{1} \rfloor}\mu_{i}+k_{2}(k_1-\lfloor k_{1} \rfloor)\mu_{\lfloor k_{1} \rfloor+1}\,.
\end{equation}
By the non-increasing property of $\{\mu_{i}\}_{i=1}^{N}$ and
\[
k_{1}(\lfloor k_{2} \rfloor -\lfloor k_{1} \rfloor )+k_{1}(k_2-\lfloor k_{2} \rfloor)=(k_2-k_1)\lfloor k_{1} \rfloor+k_{2}(k_1-\lfloor k_{1} \rfloor)=k_1(k_2-\lfloor k_{1} \rfloor)\,,
\]
we have
\begin{align*}
\textrm{LHS of }\, \eqref{monotone}\geq & \, k_1(k_2-\lfloor k_{1} \rfloor)\mu_{\lfloor k_{1} \rfloor+1}\,,\\[2mm]
\textrm{RHS of }\, \eqref{monotone}\leq & \, k_1(k_2-\lfloor k_{1} \rfloor)\mu_{\lfloor k_{1} \rfloor+1}\,.
\end{align*}
Combining these two inequalities, yields the desired inequality \eqref{monotone}.
\end{proof}

\subsection{Non-negativity of the shifted cone $\Gamma^{+}_{2}(\alpha)$}
In this subsection, we investigate the non-negativity properties of the shifted cone $\Gamma^{+}_{2}(\alpha)$ for parameters $0<\a<\frac{1}{N}$. For our analysis, we fix $\epsilon\in(0,1)$ and define the key parameters:
\begin{equation}\label{alpha m}
\alpha_\epsilon=\frac{1}{N}(1-\epsilon), \quad
m_\epsilon = \frac{N(N-1)\epsilon^2}{1+(N-1)\epsilon^2}
\end{equation}
\begin{remark}
The function $m_{\epsilon}$ is strictly increasing in $\epsilon\in(0,1)$ and the range satisfies $m_{\epsilon}\in(0, N-1)$.
\end{remark}
The $\a_\epsilon$-shifted vector of $V$ is given by:
$$V_{\alpha_{\epsilon}}=V-\alpha_{\epsilon}\,\overline{\mu}\,I_{N}=V-\alpha_{\epsilon}\,(\overline{\mu},\cdots,\overline{\mu}),$$
where $I_N = (1,...,1)\in\mathbb{R}^{N}$ and $\bar{\mu}=\sum_{i=1}^{N}\mu_{i}$. For every integer $1\leq k\leq N-1$, Yang and Zhang \cite{YZ25} considered the shifted vector $V_{\a_{k}}$ with
$$\alpha_{k}=\frac{1}{N}\Big(1-\sqrt{\frac{k}{(N-1)(N-k)}}\Big).$$
They proved the following non-negatively of cone $\Gamma^{+}_{2}(\a_{k})$.
\begin{proposition}(\cite[Lemma 2.4]{YZ25})
For $1\leq k\leq N-1$, if $V\in\overline{\Ga}^{+}_{2}(\a_{k})$, then
$$\mu_{1}+\mu_{2}+\cdots+\mu_{k}\geq0.$$
Moreover, exactly one of the followings holds.\\
(1) $\mu_{1}+\mu_{2}+\cdots+\mu_{k}>0$.\\
(2) $\mu_{1}=\mu_{2}=\cdots=\mu_{k}=0$ and $\mu_{k+1}=\mu_{k+2}=\cdots=\mu_{N}\geq0$.
\end{proposition}
For any constant $\epsilon\in(0,1)$, we establish the following inclusion relation through an approach distinct from others.
\begin{proposition}\label{cone inclusion}
For every $0<\epsilon<1$, we have the inclusion relation:
\begin{equation}
\overline{\Gamma}^{+}_{2}(\alpha_{\epsilon})\subset \overline{\mathcal{P}}_{m_{\epsilon}}\,.
\end{equation}
\end{proposition}
\begin{proof}
The proof consists of several steps.\medskip

\noindent\textbf{Step 1.}  Fix an arbitrary vector $V$ with $V_{\alpha_{\epsilon}}\in \overline{\Gamma}^{+}_{2}$. Through direct computation, we establish 
\begin{equation*}
\begin{split}
2\sigma_{2}(V_{\alpha_{\epsilon}})&=\big(\sigma_{1}(V_{\alpha_{\epsilon}})\big)^2-\sum_{i=1}^{N}(\mu_{i}-\alpha_{\epsilon}\overline{\mu})^2\\
&=(1-N\alpha_{\epsilon})^{2}\overline{\mu}^{2}-\sum_{i=1}^{N}\big(\mu_{i}^{2}-2\alpha_{\epsilon}\overline{\mu}\mu_{i}+\alpha_{\epsilon}^{2}\overline{\mu}^{2} \big)\\
&=\big[(1-N\alpha_{\epsilon})^{2}+2\alpha_{\epsilon}+N\alpha_{\epsilon}^{2} \big]\overline{\mu}^{2}-\sum_{i=1}^{N}\mu_{i}^{2}\,.
\end{split}
\end{equation*}
For the first term, using the definition of $\alpha_{k}$, we derive
\begin{equation*}
\begin{split}
&(1-N\alpha_{\epsilon})^{2}+2\alpha_{\epsilon}+N\alpha_{\epsilon}^{2}\\
&=N(N-1)\alpha_{\epsilon}^{2}-2(N-1)\alpha_{\epsilon}+1\\
&=N(N-1)\big[\alpha_{\epsilon}^{2}-\frac{2}{N}\alpha_{\epsilon}+\frac{1}{N(N-1)} \big]\\
&=N(N-1)\big[(\alpha_{\epsilon}-\frac{1}{N})^{2}+\frac{1}{N^{2}(N-1)} \big]\\
&=N(N-1)\big[\frac{\epsilon^{2}}{N^{2}}+\frac{1}{N^{2}(N-1)} \big]\\
&=\frac{1+(N-1)\epsilon^2}{N}\,.
\end{split}
\end{equation*}
This establishes the fundamental inequality 
\begin{equation}\label{sigma 2}
\frac{1+(N-1)\epsilon^2}{N}\overline{\mu}^{2}\geq \sum_{i=1}^{N}\mu_{i}^{2}\,.
\end{equation}

\noindent\textbf{Step 2.} We first assume $\overline{\mu}=0$.
Since $\sigma_{2}(V_{\alpha_{\epsilon}})\geq 0$,  the inequality (\ref{sigma 2}) forces $\mu_{i}=0$ for all $i$, implying $V=0$. Thus $V\in \overline{\mathcal{P}}_{m_{\epsilon}}$ trivially holds. Consequently, we may assume $\overline{\mu}>0$ in subsequent arguments. 

\noindent\textbf{Step 3.} We demonstrate that any vector $V = (\mu_{1}, \cdots, \mu_{N})$ satisfying both $\overline{\mu}>0$ and \eqref{sigma 2} belongs to $\overline{\mathcal{P}}_{m_{\epsilon}}$.
\begin{proof}[Proof of Claim]

Let $m = m_{\epsilon}$ and assume without loss of generality that $\mu_{1} \leq \cdots \leq \mu_{N}$. 
 It is sufficient to prove
\begin{equation}\label{k positive'}
c_{0}:=\sum_{i=1}^{\lfloor m \rfloor}\mu_{i}+\big(m-\lfloor m \rfloor\big)\mu_{\lfloor m \rfloor+1}\geq 0\,.
\end{equation}
Decompose the right-hand side of \eqref{sigma 2} using the index $[m]$
$$T_{1}:=\sum_{i=1}^{[m]}\mu_{i}^{2}+\big(m-[m]\big)\mu_{[m+1]}^{2} \quad \textrm{ and } \quad T_{2}:=\sum_{i=[m+2]}^{N}\mu_{i}^{2}+\big([m+1]-m\big)\mu_{[m+1]}^{2}.$$
Denote vectors in $\mathbb{R}^{N - \lfloor m \rfloor}$:
$$ \alpha=(1,\cdots,1,\sqrt{[m+1]-m})\,, \quad \beta=(\mu_{[m+2]}, \cdots, \mu_{N}, \sqrt{[m+1]-m}\mu_{[m+1]}).$$ 
These satisfy
$$|\alpha|^2=N-m\,, \qquad |\beta|^2=T_{2}\,, \qquad \langle \alpha,\beta\rangle=\overline{\mu}-c_{0}.$$
Applying the Cauchy-Schwarz inequality $|\langle \alpha,\beta\rangle|^{2}\leq |\alpha|^{2}|\beta|^{2}$ yields: 
$$T_{2}\geq \frac{(\overline{\mu}-c_{0})^{2}}{N-m}=\frac{1+(N-1)\epsilon^2}{N}(\overline{\mu}-c_{0})^{2}.$$
Substituting into \eqref{sigma 2} gives:
$$\frac{1+(N-1)\epsilon^2}{N}\overline{\mu}^{2}\geq T_{1}+\frac{1+(N-1)\epsilon^2}{N}(\overline{\mu}-c_{0})^{2}.$$
Consequently, we obtain the key inequality
\begin{equation}\label{E1}
\frac{2 + 2(N-1)\epsilon^2}{N}\overline{\mu}c_{0} \geq T_{1} + \frac{1 + (N-1)\epsilon^2}{N}c_{0}^{2} \geq 0.
\end{equation}
The non-negativity of $c_{0}$ follows from $\overline{\mu} > 0$, establishing \eqref{k positive'}. 
\end{proof}
The preceding arguments complete the proof of Proposition \ref{cone inclusion}.
\end{proof}
Let us now proceed to prove Theorem \ref{T1} and Corollary 1.2.
\begin{proof}[\textbf{Proof of Theorem \ref{T1}}]
Proposition \ref{cone inclusion} yields that
$$\mu_{1}+\cdots+\mu_{\lfloor  m\rfloor}+\big(m-\lfloor m\rfloor\big)\mu_{\lfloor m\rfloor+1}\geq0.$$
If equality holds, then \eqref{E1} implies
$$T_{1}=\sum_{i=1}^{\lfloor m\rfloor}\mu_{i}^{2}+\big(m-\lfloor m\rfloor\big)\mu_{\lfloor m\rfloor+1}^{2}=0,$$
which necessitates
$$\mu_{1}=\cdots=\mu_{\lfloor m\rfloor}=\big(m-\lfloor m\rfloor\big)\mu_{\lfloor m\rfloor+1}=0.$$
From \eqref{sigma 2}, we derive
\begin{equation*}
\frac{1}{N-m}\overline{\mu}^{2}\geq\sum_{i=\lfloor m\rfloor+2}^{N}\mu_{i}^{2}+\big(\lfloor m\rfloor+1-m\big)\mu^{2}_{\lfloor m\rfloor+1}.
\end{equation*}
When $m \notin \mathbb{N}$, we have $\mu_{\lfloor m \rfloor + 1} = 0$ (since $m\in(1,N-1)$) and $\lfloor m \rfloor + 1 \in [1, N-1]$, leading to
\begin{equation*}
\begin{split}
\frac{1}{N-m}\overline{\mu}^{2}&\geq\sum_{i=\lfloor m\rfloor+2}^{N}\mu_{i}^{2}\\
&\geq\frac{1}{N-(\lfloor m\rfloor+1)}\Big(\sum_{i=\lfloor m\rfloor+2}^{N}\mu_{i}\Big)^2\\
&=\frac{\bar{\mu}^{2}}{N-(\lfloor m\rfloor+1)}.
\end{split}
\end{equation*}
This implies $\overline{\mu} = 0$, contradicting our initial assumption. Therefore, $m$ must be a positive integer.

For $m \in \mathbb{N}^{+}$, we conclude $\mu_{1} = \cdots = \mu_{m} = 0$, and by similar reasoning, $\mu_{m+1} = \cdots = \mu_{N} > 0$. Furthermore, $\sigma_{2}(\alpha_{\epsilon})=0$ when $c_{0}=0$.
\end{proof}
\begin{proof}[\textbf{Proof of Corollary 1.2}]
For every $V= (\mu_{1}, \dots, \mu_{N}) \in \Gamma^{+}_{\alpha_{\varepsilon}} $, i.e., satisfying $\sigma_{2}(V_{\alpha_{\varepsilon}}) > 0$ and $ \sum_{i=1}^{N} \mu_{i} > 0$, Theorem \ref{T1} implies that $V$ must satisfy one of the following conditions.
\begin{itemize}
\item[(i)] Either the inequality
$$\mu_{1}+\cdots+\mu_{\lfloor  m_{\epsilon}\rfloor}+\big(m_{\epsilon}-\lfloor m_{\epsilon}\rfloor\big)\mu_{\lfloor m_{\epsilon}\rfloor+1}>0$$
holds, meaning $V \in \mathcal{P}_{m_{\varepsilon}}$.
\item[(ii)] Or $m_{\epsilon} \in \mathbb{N}^{+}$, $\mu_{1} = \mu_{2} = \cdots = \mu_{m_{\epsilon}} = 0$, and $\mu_{m_{\epsilon} + 1} = \mu_{m_{\epsilon} + 2} = \cdots = \mu_{N} > 0$. Then we have  $\sigma_{2}(\alpha_{\varepsilon}) = 0$.
\end{itemize}
Consequently, we conclude that $\Gamma^{+}_{2}(\alpha_{\epsilon}) \subset \mathcal{P}_{m_{\epsilon}}$.

\end{proof}

\section{Positivity of various curvature operators}
\subsection{Curvature operator}
Let $(X,g)$ be a closed $n$-dimensional Riemannian manifold ($n \geq 3$), and let $\mathfrak{X}(X)$ denote the space of smooth vector fields on $X$. The Riemannian connection $\nabla : \mathfrak{X}(X) \times \mathfrak{X}(X) \to \mathfrak{X}(X)$ defines the curvature tensor $R$ of the Riemannian metric $g$ by
$$
R(U,V)W = \nabla_V \nabla_U W - \nabla_U \nabla_V W + \nabla_{[U,V]} W
$$
for all $U,V,W \in \mathfrak{X}(X)$. The Weitzenb\"{o}ck curvature operator, denoted by $\mathfrak{Ric}$ throughout this paper, is the zeroth-order curvature term appearing in the classical Weitzenb\"ock formula. For a smooth $k$-tensor $T$ on $X$, $\mathfrak{Ric}$ is defined as
\[
\mathfrak{Ric}(T)(X_{1},\dots,X_{k}) = \sum_{i=1}^{k} \sum_{j=1}^{n} \left( R(X_{i}, e_{j}) T \right)(X_{1}, \dots, \underset{\substack{\uparrow \\ i\text{-th}}}{e_j}, \dots, X_{k}),
\]
where $\{e_j\}_{j=1}^n$ is a local orthonormal frame of $TX$, and $X_1, \dots, X_k \in \mathfrak{X}(X)$ are smooth vector fields. This operator generalizes the classical Ricci tensor, reducing to standard Ricci curvature when applied to vector fields or $1$-forms.
 
Recall that the Hodge Laplacian is defined as 
\[
\Delta_{d} = d d^* + d^* d.
\]
For a smooth $k$-form $\omega$ on $X$, the classical Weitzenb\"ock formula states
\[
\Delta_{d} \omega = \nabla^* \nabla \omega + \mathfrak{Ric}(\omega),
\]
where $\nabla^* \nabla$ denotes the connection Laplacian. When $\omega$ is harmonic (i.e. $\Delta_{d} \omega = 0$), it satisfies the following Bochner identity 
\begin{equation}
\frac{1}{2} \Delta_{d} |\omega|^2 = |\nabla \omega|^2 + \langle \mathfrak{Ric}(\omega), \omega \rangle.
\end{equation}
For a detailed introduction on the Weitzenb\"ock curvature operator, see \cite[Chapter 9]{Pet16}.

Recall that the curvature operator $\mathfrak{R} \colon \wedge^2 TX \to \wedge^2 TX$ acting on $2$-forms is defined by
\[
(\mathfrak{R}(\omega))_{ij} = \sum_{k,l=1}^{n} R_{ijkl} \omega_{kl},
\]
Let $\Lambda=(\lambda_{1},\cdots,\lambda_{N_{1}})$  be the ordered eigenvalues of $\mathfrak{R}$. Then the scalar curvature $R$ of the metric $g$ satisfies \eqref{R1}. For a fixed $\a\in [0,\frac{1}{N})$, we define a shifted vector in $\mathbb{R}^{N}$ via
\begin{equation*}
\Lambda_{\alpha}:=\La-\a\Big(\frac{R}{2},\frac{R}{2},\cdots,\frac{R}{2}\Big).
\end{equation*}
Note that $\Lambda_{0}=\Lambda$ recovers the original curvature eigenvalues.
\begin{itemize}
\item[(i)] First Shifted Cone Condition $\Gamma^{+}_{1}(\alpha)$:
A metric $g$ belongs to $\Gamma^{+}_{1}(\alpha)$ if and only if
\begin{equation}\label{E5}
\sigma_{1}(V_{\alpha})=\sum_{i=1}^{N}\Big(\lambda_{i}-\frac{\a R}{2}\Big)=\frac{(1-\a N)}{2}R>0.
\end{equation}
\item[(ii)] Second Shifted Cone Condition $\Gamma^{+}_{2}(\alpha)$:
The stricter condition $g\in\Gamma^{+}_{2}(\alpha)$ requires additionally that
\begin{equation}\label{E6}
\sigma_{2}(V_{\alpha})=\sum_{1\leq i<i\leq N}\Big(\lambda_{i}-\frac{R\alpha}{2}\Big)\Big(\lambda_{j}-\frac{R\alpha}{2}\Big)>0.
\end{equation}
\end{itemize}
We demonstrate that for certain geometrically meaningful values of $\a$, the positivity conditions $\sigma_{1}(V_{\alpha})>0$ and $\sigma_{2}(V_{\alpha})>0$ encode significant geometric information about $(X,g)$.

From equation (\ref{E5}), we immediately infer that the scalar curvature $R$ must be non-negative. Of particular interest is the parameter choice $\a_{\epsilon}=\frac{1}{N}(1-\epsilon)$, where $\epsilon\in(0,1)$. This leads us to study the curvature positivity condition defined by the shifted cone $\Gamma^{+}_{2}(\a_{\epsilon})$. 
\begin{theorem}\label{T2}
Let $(X,g)$ be a compact Riemannian manifold of dimension $n\geq 3$. If
$g\in\Gamma^{+}_{2}(\a_{\epsilon})$ for some $\epsilon\in(0,1)$, then the curvature operator $\mathfrak{R}$ of $(X,g)$ is $m_{\epsilon}$-positive, i.e., 
$$\lambda_{1}+\cdots+\lambda_{\lfloor m_{\epsilon}\rfloor}+(m_{\epsilon}-\lfloor m_{\epsilon}\rfloor)\lambda_{\lfloor m_{\epsilon}\rfloor+1 }>0.$$
Moreover, the Betti numbers of $X$ satisfy the following statements.
\begin{itemize}
\item[(i)] When $k-1<m_{\epsilon}\leq k$  for some integer $1\leq k\leq \lceil\frac{n}{2}\rceil$, then
 $$b_{p}(X)=0\ {\rm{for\ all}}\ 1\leq p\leq n-1.$$
\item[(ii)] When $k-1<m_{\epsilon}\leq k$  for some integer $ \lceil\frac{n}{2}\rceil+1\leq k\leq n-1$, then
$$b_{1}(X)=\cdots=b_{n-k}(X)=0\quad and\quad b_{k}(X)=\cdots=b_{n-1}(X)=0.$$
\end{itemize}
\end{theorem}
\begin{proof}
If $g\in\Gamma^{+}_{2}(\a_{\epsilon})$ for some $\epsilon\in(0,1)$, then by Corollary \ref{C2}, the curvature operator $\mathfrak{R}$ is  $m_{\epsilon}$-positive. Based on the Theorem A in \cite{PW21a}, we obtain the following topological restrictions:
\begin{itemize}
\item[Case (i).] If $k-1<m_{\epsilon}\leq k$  for some integer $1\leq k\leq \lceil\frac{n}{2}\rceil$, then the curvature operator is $\lceil\frac{n}{2}\rceil$-positive (cf. Lemma \ref{L1}). Therefore, we have $b_{p}(X)=0$ for all $1\leq p\leq n-1$.
\item[Case (ii).]  If $k-1<m_{\epsilon}\leq k$ for some integer $ \lceil\frac{n}{2}\rceil+1\leq k\leq n-1$, then the curvature operator is $k$-positive. Thus, the Betti numbers satisfy 
$$b_{1}(X)=\cdots=b_{n-k}(X)=0\quad and\quad b_{k}(X)=\cdots=b_{n-1}(X)=0.$$
\end{itemize}
Then the proof is completed.
\end{proof}
We have the following characterization of spherical space forms via the curvature operator.
\begin{theorem}\label{T3}
Let $(X,g)$ be a compact Riemannian manifold of dimension $n\geq3$. If the metric $g$ satisfies the curvature condition $g\in\Gamma^{+}_{2}(\a_{\epsilon})$ for some  $0<\epsilon\leq\sqrt{\frac{2}{(N_{1}-1)(N_{1}-2)}},$ 
then $X$ is diffeomorphic to a spherical space form.
\end{theorem}
\begin{proof}
For every $\epsilon$ satisfying $0<\epsilon\leq\sqrt{\frac{2}{(N_{1}-1)(N_{1}-2)}},$  we have
\[
m_{\epsilon}=\frac{N_{1}(N_1-1)\epsilon^2}{1+(N_1-1)\epsilon^2}\leq 2.
\]
Since $g\in\Gamma^{+}_{2}(\a_{\epsilon})$, Theorem \ref{T1} implies that the curvature operator $\mathfrak{R}$  is $m$-positive for all $0<m\leq2$. This particularly implies $\lambda_{1}+\lambda_{2}>0$. Therefore, by the classification theorems of Chen \cite{Che91} and B\"{o}hmann and Wilking \cite{BW08}, the conclusion follows.
\end{proof}

\subsection{Curvature operator of the second kind}
Let $(X,g)$ be a compact Riemannian manifold of dimension $n\geq 3$. Let $\Lambda_{2}=(\nu_{1},\cdots,\nu_{N_2})$ denote the ordered eigenvalues of the curvature operator of the second kind $\mathcal{R}$, and let $\{S_{i}\}_{i=1}^{N_2}$ be its corresponding orthonormal eigenvectors. For an arbitrary smooth $p$-form $T$ and any $S\in S^{2}_{0}(TX)$, the action of $S$ on $T$ is defined by
$$(ST)(X_{1},\cdots,X_{p})=\sum_{i=1}^{p}T(X_{1},\cdots,SX_{i},\cdots,X_{p}).$$ 
Then the decomposition of the Weitzenb\"{o}ck curvature operator $\mathfrak{Ric}$ acting on $T$ is given by 
\begin{equation}\label{decom}
\frac{3}{2}g(\mathfrak{Ric}(T),T)=\sum_{i=1}^{N_{2}}\nu_{i}|S_{i}T|^{2}+\frac{p(n-2p)}{n}\sum_{j=1}^{n}g(i_{\mathfrak{Ric}(e_{i})}T,i_{e_{j}}T)+\frac{p^{2}}{n^2}R|T|^{2}\,,
\end{equation}
as established in \cite{NPW23,NPWW23}. 
In order to  control the first term in \eqref{decom} by understanding the interaction of trace-free, symmetric tensors on forms, the authors developed a powerful calculus in \cite{NPW23} to estimate any finite weighted sums  with nonnegative weights.
\begin{definition}(\cite[Definition 3.1]{NPW23})
Let $\{\omega_{i}\}_{i=1}^{N_2}$ be the nonnegative weights of any finite weighted sums. Define
\[
\Om=\max_{1\leq i\leq N_2}\omega_{i} \quad \mathrm{and} \quad  S=\sum_{i=1}^{N_2}\omega_{i}\,.
\]
We call $S$ the total weight and $\Om$ the highest weight.
\end{definition} 
Let $\Lambda_{2}=(\nu_{1},\cdots,\nu_{N_2})$ denote the ordered eigenvalues of the operator $\mathcal{R}$.  In what follows, we use the notation $[\mathcal{R},\Om,S]$ to denote the supremum of all finite weighted sums $\sum_{i=1}^{N_2}\omega_{i}\nu_{i}$  in terms of $\Lambda_2$ with highest weight $\Om$ and total weight $S$. 
\begin{lemma}\label{Lem 3.4}(\cite[Lemma 3.4]{NPW23})
Let $\Lambda_{2}=(\nu_{1},\cdots,\nu_{N_2})$ denote the ordered eigenvalues of $\mathcal{R}$. Then, for any integer $1\leq m\leq N_{2}$, we have
$$[\mathcal{R},\Om,S]\geq(S-m\Om)\nu_{m+1}+\Om\sum_{i=1}^{m}\nu_{i}.$$\end{lemma}
\begin{proposition}\label{Prop 3.15}(\cite[Proposition 3.15]{NPW23})
Let  $T$ be a $p$-form with $p\leq\frac{n}{2}$. Then we have
$$\frac{3}{2}g(\mathfrak{Ric}(T),T)\geq \Big[\mathcal{R},\frac{n^{2}p-np^{2}-2np+2n^{2}+2n-4p}{n(n+2)},\frac{3}{2}p(n-p) \Big]\cdot |T|^{2}.$$
\end{proposition}
For notational convenience, for every integer $p\leq \frac{n}{2}$, we introduce a positive coefficient function $C_{p}$ as defined in \cite{PW21b}:
\begin{equation}\label{Cp}
\begin{split}
C_{p}&=\frac{3}{2}\frac{n(n+2)p(n-p)}{n^{2}p-np^{2}-2np+2n^{2}+2n-4p}=\frac{3}{2}\frac{n(n+2)}{n+\frac{2n}{p}+\frac{2n-4p}{p(n-p)}}.\\
\end{split}
\end{equation}

\begin{corollary}\label{C1}
Let $\Lambda_{2}=(\nu_{1},\cdots,\nu_{N_{2}})$  be the ordered eigenvalues of the operator  $\mathcal{R}$. Assume that 
\begin{equation}\label{C1 assume 1}
\nu_1 + \cdots + \nu_{\lfloor C_{p} \rfloor} + (C_{p} - \lfloor C_{p} \rfloor)\nu_{\lfloor C_{p} \rfloor+1} \geq C_{p}\,\kappa
\end{equation}
for some integer $p\leq \frac{n}{2}$,  then for any smooth $k$-form $T$ on $X$ with $p\leq k\leq \frac{n}{2}$, we have
\begin{equation*}
g(\mathfrak{Ric}(T), T) \geq \kappa k(n-k)|T|^2.
\end{equation*}
In particular,  assume that
\begin{equation}\label{C1 assume 2}
\nu_{1}+\cdots+\nu_{\lfloor\frac{3n}{4}\rfloor }+\Big(\frac{3n}{4}-\lfloor\frac{3n}{4}\rfloor\Big)\nu_{\lfloor\frac{3n}{4}\rfloor +1}\geq \frac{3n}{4}\kappa\,,
\end{equation}
then for any smooth $p$-form $T$ with $p\leq \frac{n}{2}$, we have  
\begin{equation*}
g(\mathfrak{Ric}(T),T)\geq\kappa p(n-p)|\a|^{2}.
\end{equation*}
\end{corollary}
\begin{proof}
For every $\Omega$ and $S$,  \cite[Lemma 3.3]{NPW23} and Lemma \ref{Lem 3.4} yield the following estimate 
\begin{equation}\label{tri}
\begin{split}
[\mathcal{R},\Om,S]&=\Om\cdot \Big[\mathcal{R},1,\frac{S}{\Om}\Big]\\
&\geq \Om\cdot\Big(\big(\frac{S}{\Om}-\lfloor \frac{S}{\Om}\rfloor\big)\nu_{\lfloor \frac{S}{\Om}\rfloor+1}+\sum_{i=1}^{\lfloor \frac{S}{\Om}\rfloor}\nu_{i}\Big)\\
&=S\frac{\sum_{i=1}^{\lfloor \frac{S}{\Om}\rfloor}\nu_{i}+\big(\frac{S}{\Om}-\lfloor \frac{S}{\Om}\rfloor\big)\nu_{\lfloor \frac{S}{\Om}\rfloor+1}}{\frac{S}{\Om}}.\\
\end{split}
\end{equation}
Set 
\[
\Omega=\frac{n^2p-np^2-2np+2n^2+2n-4p}{n(n+2)} \quad \textrm{and} \quad S=\frac{3}{2}p(n-p)  
\]
so that $ \frac{S}{\Om}=C_{p}$. By the definition of $C_p$ in \eqref{Cp}, we see that $C_{k}\geq C_{p}$ for all $p\leq k\leq \frac{n}{2}$.  Then by Lemma \ref{L1} and \eqref{C1 assume 1}, we obtain 
\begin{equation*}
\left[\mathcal{R}, \frac{n^2k-nk^2-2nk+2n^2+2n-4k}{n(n+2)}, \frac{3}{2}k(n-k)\right] \geq \frac{3}{2}k(n-k)\kappa.
\end{equation*}
This proves the first statement by Proposition \ref{Prop 3.15}.

Now we proceed to prove the second statement. It is straightforward to verify that for all $n\geq 3$
\[
C_1=\frac{3n}{2}\frac{n^2+n-2}{3n^2-n-4}\leq \frac{3n}{4}\,, \quad C_2=\frac{3n(n+2)(n-2)}{4n^2-6n-8}>\frac{3n}{4} 
\]
and $C_{p}$ is monotonically increasing in the parameter $p$. We divide the proof into two cases based on $p$:
\begin{itemize}
\item[(i)] If $p=1$. According to \cite[Remark 3.16]{NPW23}, we have the improved estimate
\begin{equation*}
\frac{3}{2}g(\mathfrak{Ric}(T),T) \geq \left[\mathcal{R}, \frac{2n-1}{n+2}, \frac{3(n-1)}{2}\right] \cdot |T|^2.
\end{equation*}
Set $C'_{1}=\frac{3(n-1)}{2} \cdot \frac{n+2}{2n-1}\geq \frac{3n}{4}$. Then by \eqref{tri}, we obtain
\[
\left[\mathcal{R}, \frac{2n-1}{n+2}, \frac{3(n-1)}{2}\right]\geq \frac{3}{2}p(n-p)\frac{\sum_{i=1}^{\lfloor C'_{1}\rfloor}\nu_{i}+\big(C'_{1}-\lfloor C'_{1}\rfloor\big)\nu_{\lfloor C'_{1}\rfloor+1}}{C'_{1}}\,. 
\]
Using Lemma \ref{L1}, we have
\[
\left[\mathcal{R}, \frac{2n-1}{n+2}, \frac{3(n-1)}{2}\right]\geq \frac{3}{2}p(n-p)\frac{\sum_{i=1}^{\lfloor \frac{3n}{4}\rfloor}\nu_{i}+\big(\frac{3n}{4}-\lfloor \frac{3n}{4}\rfloor\big)\nu_{\lfloor \frac{3n}{4}\rfloor+1}}{\frac{3n}{4}}\geq \frac{3}{2}p(n-p)\kappa\,. 
\]
This proves the case $p=1$.
\item[(ii)] If $p\geq 2$. In this setting $C_p>\frac{3n}{4}$. Then by \eqref{tri} and Lemma \ref{L1} again, we obtain
\[
\begin{split}
& \Big[\mathcal{R},\frac{n^{2}p-np^{2}-2np+2n^{2}+2n-4p}{n(n+2)},\frac{3}{2}p(n-p) \Big]\\
& \geq  \frac{3}{2}p(n-p)\frac{\sum_{i=1}^{\lfloor C_{p}\rfloor}\nu_{i}+\big(C_{p}-\lfloor C_{p}\rfloor\big)\nu_{\lfloor C_{p}\rfloor+1}}{C_{p}}\\
& \geq  \frac{3}{2}p(n-p)\frac{\sum_{i=1}^{\lfloor \frac{3n}{4}\rfloor}\nu_{i}+\big(\frac{3n}{4}-\lfloor \frac{3n}{4}\rfloor\big)\nu_{\lfloor \frac{3n}{4}\rfloor+1}}{\frac{3n}{4}}\\
& \geq \frac{3}{2}p(n-p)\kappa\,. 
\end{split}  
\]
This proves the case $p\geq  2$.
\end{itemize}
Consequently, we complete the proof of the corollary.
\end{proof}

We now introduce a novel shifted cone construction for Riemannian manifolds based on the curvature operator of the second kind $\mathcal{R}$. Let $(X,g)$ be a compact Riemannian manifold of dimension $n \geq 3$, and consider the ordered eigenvalues $\Lambda_{2}=(\nu_{1},\cdots,\nu_{N_{2}})$ of $\mathcal{R}$. The scalar curvature $R$ of the Riemannian metric $g$ satisfies \eqref{R2}. For any non-negative parameter $\beta\geq 0$, we define the $\beta$-shifted curvature operator as follows:
\begin{equation*}
\Lambda_{\beta}:=\Lambda_{2}-\b\Big(\frac{(n+2)R}{2n},\frac{(n+2)R}{2n},\cdots,\frac{(n+2)R}{2n}\Big).
\end{equation*}
\begin{theorem}\label{T4}
Let $(X,g)$ be a compact Riemannian manifold of dimension $n \geq 3$. If $X \in \Gamma^{+}_2(\beta_\epsilon)$ for some $\epsilon \in (0,1)$, then the curvature operator of the second kind $\mathcal{R}$ is $m_{\epsilon}$-positive, i.e., 
\begin{equation*}
\nu_1 + \cdots + \nu_{\lfloor m_{\epsilon} \rfloor} + (m_{\epsilon} - \lfloor m_{\epsilon} \rfloor)\nu_{\lfloor m_{\epsilon} \rfloor+1} > 0.
\end{equation*}
Furthermore, the Betti numbers of $X$ satisfy the following statements.
\begin{itemize}
\item[(i)]  If $m_{\epsilon} \leq \frac{3n}{4}$, then all Betti numbers vanish. Namely,
\begin{equation*}
b_p(X) = 0, \quad {\rm{for\ all}}\ 1 \leq p \leq n-1.
\end{equation*}
\item[(ii)]  For any integer $1 \leq p \leq \lfloor \frac{n}{2} \rfloor$, if $m_{\epsilon} \leq C_p(n)$, then
\begin{equation*}
b_p(X) = \cdots = b_{n-p}(X)=0.
\end{equation*}
\end{itemize}
\end{theorem}
\begin{proof}
If $g\in\Gamma^{+}_{2}(\beta_{\epsilon})$ for some $\epsilon\in(0,1)$,  Theorem \ref{T1} implies that the
curvature operator of the second kind $\mathcal{R}$ is $m_{\epsilon}$-positive, i.e., there exists a positive constant $\kappa$ such that
$$\nu_{1}+\cdots+\nu_{\lfloor m_{\epsilon}\rfloor}+(m_{\epsilon}-\lfloor m_{\epsilon}\rfloor)\nu_{\lfloor m_{\epsilon}\rfloor+1}\geq m_{\epsilon}\kappa.$$
Corollary \ref{C1} yields the following topological constraints:
\begin{itemize}
\item[(i)] If $m_{\epsilon}\leq\frac{3n}{4}$, then the curvature operator of the second is $\frac{3n}{4}$-positive, and consequently $b_{p}(X)=0$ for all $1\leq p\leq n-1$.
\item[(ii)] If $m_{\epsilon} \leq C_p$, then $\mathcal{R}$ is $C_p$-positive, which implies that  $$b_{p}(X)=\cdots=b_{\lfloor\frac{n}{2}\rfloor}(X)=b_{\lceil\frac{n}{2}\rceil}(X)=\cdots=b_{n-p}(X)=0.$$
\end{itemize}
This ends the proof of the theorem.
\end{proof}
Finally, we have the following characterization of spherical space forms via the curvature operator of the second kind $\mathcal{R}$.
\begin{theorem}\label{T5}
Let $(X,g)$ be a compact Riemannian manifold of dimension $n\geq3$. If $g\in\Gamma^{+}_{2}(\beta_{\epsilon})$ for some $0<\epsilon\leq\sqrt{\frac{3}{(N_{2}-1)(N_{2}-3)}},$
then $X$ is diffeomorphic to a spherical space form.
\end{theorem}
\begin{proof}
For every $\epsilon$ satisfying $0<\epsilon\leq\sqrt{\frac{3}{(N_{2}-1)(N_{2}-3)}},$  we have
\[
m_{\epsilon}=\frac{N_{2}(N_2-1)\epsilon^2}{1+(N_2-1)\epsilon^2}\leq 3.
\]
Since  $g\in\Gamma^{+}_{2}(\beta_{\epsilon})$, Theorem \ref{T1} implies that the curvature operator of the second kind $\mathcal{R}$ is $m$-positive for all $0<m\leq 3$, and in particular,
\[
\nu_1+\nu_2+\nu_3>0.
\]
The conclusion then follows from a result of Li \cite[Theorem 1.1]{Li24}, which establishes that Riemannian manifolds with $3$-positive curvature operator of the second kind must be spherical space forms.
\end{proof}

\subsection{Induced K\"{a}hler curvature operator}
Recall that for a K\"{a}hler manifold $(X,g)$ of complex dimension $n$, the curvature operator $\mathfrak{R}$ vanishes on the orthogonal complement of the holonomy algebra $\mathfrak{u}(n)\subset\mathfrak{so}(2n)$. This motivates us to study the induced K\"{a}hler curvature operator $\mathscr{R}:\mathfrak{u}(n)\rightarrow\mathfrak{u}(n)$ with ordered eigenvalues $\Lambda_{3}=(\rho_{1},\cdots,\rho_{N_{3}})$, where \begin{equation}
N_{3}=n^2=\dim \mathfrak{u}(n)\,.
\end{equation} 

For every $\epsilon\in(0,1)$, define $\gamma_{\epsilon}=\frac{1}{N_{3}}(1-\epsilon)$. The $\gamma_{\epsilon}$-shifted vector of $\Lambda_{3}$ is given by
$$\Lambda_{\gamma_{\epsilon}}=\Lambda_{3}-\gamma_{\epsilon}\Big(\sum_{i=1}^{N_{3}}\rho_{i},\cdots,\sum_{i=1}^{N_{3}}\rho_{i}\Big).$$
We say the Riemannian metric $g\in\Gamma^{+}_{2}(\gamma_{\epsilon})$ if $\Lambda_{\gamma_{\epsilon}}\in\Gamma^{+}_{2}$ at every point $x\in X$. 
\begin{theorem}
Let $(X,g)$ be a compact K\"{a}her manifold of complex dimension $n\geq 2$. If $g\in\Gamma^{+}_{2}(\gamma_{\epsilon})$ for some $0<\epsilon\leq  \sqrt{\frac{3n-2}{(n^{3}-3n+2)(n^{2}-1)}}$,  then $X$ has the rational cohomology ring of $\mathbb{CP}^{n}$.
\end{theorem}
\begin{proof}
For every $\epsilon$ satisfying $0<\epsilon\leq\sqrt{\frac{3n-2}{(n^{3}-3n+2)(n^{2}-1)}},$ we have
\[
m_{\epsilon}=\frac{N_{3}(N_3-1)\epsilon^2}{1+(N_3-1)\epsilon^2}\leq 3 - \frac{2}{n}.
\]
Since $g \in \Gamma^{+}_{2}(\gamma_{\epsilon})$, Theorem \ref{T1} implies that the K\"ahler curvature operator $\mathscr{R}$ is $m$-positive for all $0 < m \leq 3 - \frac{2}{n}$, and in particular 
\[
\rho_1+\rho_2+\Big(1-\frac{2}{n}\Big)\rho_{3}>0.
\]
Applying Theorem A of \cite{PW21b}, it follows that $X$ has the rational cohomology ring of $\mathbb{CP}^{n}$.
\end{proof}
Similarly, we can also obtain another characterization of $\mathbb{CP}^{n}$ via the curvature operator $\mathscr{R}$.
\begin{theorem}
Let $(X,g)$ be a compact K\"{a}hler manifold of complex dimension $n\geq2$. If $g\in\Gamma^{+}_{2}(\gamma_{\epsilon})$  for some
$0<\epsilon\leq\sqrt{\frac{2}{(n^2-1)(n^2-2)}},$
then $X$ is biholomorphic to $\mathbb{CP}^{n}$.
\end{theorem}
\begin{proof}
For every $\epsilon$ satisfying $0<\epsilon\leq\sqrt{\frac{2}{(N_{3}-1)(N_{3}-2)}},$  we have
\[
m_{\epsilon}=\frac{N_{3}(N_3-1)\epsilon^2}{1+(N_3-1)\epsilon^2}\leq 2.
\]
Since  $g\in\Gamma^{+}_{2}(\gamma_{\epsilon})$, Theorem \ref{T1} implies that the K\"{a}hler curvature operator  $\mathscr{R}$ is $m$-positive  with $0<m\leq2$,  and in particular $\rho_1+\rho_2>0$. Note that K\"ahler manifolds with  $2$-positive K\"ahler curvature operator have positive
orthogonal bisectional curvature, and thus are biholomorphic to $\mathbb{CP}^{n}$.
\end{proof}

  \section*{Acknowledgements}
This work is supported by the National Natural Science Foundation of China No. 12271496 (Huang) and the Youth Innovation Promotion Association CAS, the Fundamental Research Funds of the Central Universities, the USTC Research Funds of the Double First-Class Initiative. The second-name author was supported by the National Natural Science Foundation of China (Grant No. 12401065) and the Initial Scientific Research Fund of Young Teachers in Hefei University of Technology (Grant No. JZ2024HGQA0119). \medskip

\noindent\textbf{Data availability} {This manuscript has no associated data.}
\section*{Declarations}
\noindent\textbf{Conflict of interest} The author states that there is no conflict of interest.

\end{document}